\documentclass[12pt,A4]{amsart}
\usepackage{amsfonts,amsthm,amsmath}
\usepackage{rotating}
\usepackage{tikz}
\usepackage{verbatim}
\usepackage{amscd}
\usepackage[latin2]{inputenc}
\usepackage{t1enc}
\usepackage[mathscr]{eucal}
\usepackage{indentfirst}
\usepackage{graphicx}
\usepackage{graphics}
\usepackage{pict2e}
\usepackage{mathrsfs}
\usepackage{enumerate}
\usepackage[pagebackref]{hyperref}
\hypersetup{backref, pagebackref, colorlinks=true}
\usepackage{cite}
\usepackage{color}
\usepackage{epic}
\usepackage{hyperref} %this gives clickable references, which is nice
\usepackage{tkz-graph}
\numberwithin{equation}{section}
\theoremstyle{plain}
\newtheorem{theorem}{Theorem}[section]
\newtheorem{lemma}[theorem]{Lemma}

\theoremstyle{definition}
\newtheorem{Def}[theorem]{Definition}
\newtheorem{example}[theorem]{Example}

\newtheorem{remark}[theorem]{Remark}
\newtheorem{?}[theorem]{Problem}

\newcommand{\N}{\mathbb{N}}
\newcommand{\Z}{\mathbb{Z}}

\newcommand\des{\mathop{\rm des}}

\def\DES{\mathrm{DES}}
\def\st{\mathrm{st}}

\def\max{\mathrm{max}}
\def\min{\mathrm{min}}

\def\des{\mathrm{des}}
\def\Des{\mathrm{DES}}

\def\S{\mathfrak{S}}

\def\da{\mathrm{da}}

\def\Orb{\mathrm{Orb}}
\def\ST{\mathrm{ST}}
\def\VID{\mathrm{VID}}
\def\LMA{\mathrm{LMA}}
\def\LMI{\mathrm{LMI}}
\def\RMA{\mathrm{RMA}}
\def\RMI{\mathrm{RMI}}
\def\BJP{\mathrm{BJP}}
\def\AVA{\mathrm{AVA}}

\def\DIST{\mathrm{DIST}}
\def\ZERO{\mathrm{ZERO}}
\def\EMA{\mathrm{EMA}}
\def\zero{\mathrm{zero}}
\def\ema{\mathrm{ema}}
\def\turn{\mathrm{turn}}
\def\Red{\mathrm{red}}
\def\return{\mathrm{return}}
\def\segm{\mathrm{segment}}
\def\lma{\mathrm{lma}}
\def\lmi{\mathrm{lmi}}

\def\asc{\mathrm{asc}}
\def\ASC{\mathrm{ASC}}

\def\dist{\operatorname{dist}}
\def\EXPO{\operatorname{EXPO}}
\def\ides{\operatorname{ides}}
\def\max{\operatorname{max}}

\def\I{\operatorname{{\bf I}}}

\def\s{\operatorname{{\bf s}}}
\def\st{\operatorname{st}}

\def\boxit#1{\leavevmode\hbox{\vrule\vtop{\vbox{\kern.33333pt\hrule\kern1pt\hbox{\kern1pt\vbox{#1}\kern1pt}}\kern1pt\hrule}\vrule}}

\usepackage{collectbox}


\begin{document}

\title[Sextuple Equidistribution]{A sextuple equidistribution arising in\\ Pattern Avoidance}
\author[Z. Lin]{Zhicong Lin}
\address[Zhicong Lin]{School of Science, Jimei University, Xiamen 361021,P.R. China
\& CAMP, National Institute for Mathematical Sciences, Daejeon, 34047, Republic of Korea} 
\email{lin@nims.re.kr}

\author[D. Kim]{Dongsu Kim}
\address[Dongsu Kim]{Department of Mathematical Sciences, Korea Advanced Institute of Science and Technology, Daejeon, 34141, Republic of Korea} 
\email{dongsu.kim@kaist.ac.kr}

\date{\today}

\begin{abstract} We construct an intriguing bijection between $021$-avoiding inversion sequences and $(2413,4213)$-avoiding permutations, which proves a sextuple equidistribution involving double Eulerian statistics. Two interesting applications of this result are also presented. Moreover, this result inspires us to characterize all permutation classes that avoid two patterns of length $4$ whose descent polynomial equals that of separable permutations. 
\end{abstract}

%\thanks{Partially supported by the National Science Foundation of China grant 11501244}

\maketitle

%\tableofcontents

\section{Introduction}
Let $\s=\{s_i\}_{i\geq1}$ be a sequence of positive integers. In order to interpret the Ehrhart polynomials of the $\s$-lecture hall polytopes, Savage and Schuster~\cite{ss} introduced the associated {\em $\s$-inversion sequences of length $n$} as
$$
\I_n^{(\s)}:=\{(e_1,e_2,\ldots,e_n): 0\leq e_i<s_i\}.
$$
Since then, many remarkable equidistributions between $\s$-inversion sequences and permutations have been investigated~\cite{chen,ss,sv,lin}. In particular, Savage and Visontai~\cite{sv} established connections with the descent polynomials of Coxeter groups which enable them to settle some real root conjectures. When $\s=(1,2,3,\ldots)$, the $\s$-inversion sequences $\I_n=\I_n^{(\s)}$ are usually called {\em inversion sequences}. Recently, Corteel-Martinez-Savage-Weselcouch~\cite{cor} and Mansour-Shattuck~\cite{ms} carried out the systematic studies of patterns  in inversion sequences. A number of familiar combinatorial sequences, such as {\em large Schr\"oder numbers} and {\em Euler numbers}, arise in their studies.

 In this paper, we will prove a set-valued sextuple equidistribution arising from pattern avoidance in inversion sequences. Our result can be considered as an extention, since it involves four more statistics, of a restricted version of Foata's 1977 result~\cite{fo}, which asserts that descents and inverse descents on permutations have the same joint distribution as ascents and distinct entries on inversion sequences. We need some further notations and definitions before we can state our main result. 
 
The inversion sequences $\I_n$ serve as various kinds of codings for $\S_n$, the set of all permutations of $[n]:=\{1,2,\ldots,n\}$. By a {\em coding} of $\S_n$, we mean a bijection from $\S_n$ to $\I_n$. For example, the map $\Theta(\pi):\S_n\rightarrow\I_n$ defined for $\pi=\pi_1\pi_2\ldots\pi_n\in\S_n$ as
 $$
 \Theta(\pi)=(e_1,e_2,\ldots,e_n),\quad\text{where $e_i:=\left|\{j: \text{$j<i$ and $\pi_j>\pi_i$}\}\right|$},
 $$
 is a natural coding of $\S_n$. Clearly, the sum of the entries of $\Theta(\pi)$ equals the number of {\em inversions} of $\pi$, i.e., the number of pairs $(i,j)$ such that $i<j$ and $\pi_i>\pi_j$. This is the reason why $\I_n$ is named inversion sequences here. 
 For each $\pi\in\S_n$ and each $e\in\I_n$, let
$$
\DES(\pi):=\{i\in[n-1]:\pi_i>\pi_{i+1}\}\quad \text{and}\quad\ASC(e):=\{i\in[n-1]: e_i<e_{i+1}\}
$$ 
be the {\em {\bf\em des}cent set} of $\pi$ and the {\em {\bf\em asc}ent set} of $e$, respectively. Another important property of $\Theta$ is that $\DES(\pi)=\ASC(\Theta(\pi))$ for each $\pi\in\S_n$. Thus,
\begin{equation}\label{des:asc}
\sum_{\pi\in\S_n}t^{\DES(\pi)}=\sum_{e\in\I_n}t^{\ASC(e)},
\end{equation}
where $t^{S}:=\prod_{i\in S}t_i$ for any set $S$ of positive integers. 

Throughout this paper, we will use the convention that if  upper case  ``$\ST$'' is a set-valued statistic,
then lower case  ``$\st$'' is the corresponding numerical statistic. For example, $\des(\pi)$ is the cardinality of $\DES(\pi)$ for each $\pi\in\S_n$. It is well known that $A_n(t):=\sum_{\pi\in\S_n}t^{\des(\pi)}$ is the classical {\em$n$-th Eulerian polynomial}~\cite{fs,st} and each statistic whose distribution gives $A_n(t)$ is called a {\em Eulerian statistic}. In view of~\eqref{des:asc}, ``$\asc$'' is a Eulerian statistic on inversion sequences. For each $e\in\I_n$, define $\dist(e)$ to be the {\em number of {\bf\em dist}inct positive entries} of~$e$. This statistic was first introduced by Dumont~\cite{du}, who also showed that it is a Eulerian statistic on inversion sequences. Amazingly, Foata~\cite{fo} later invented two different codings of permutations called {\em V-code} and {\em S-code} to prove the following extension of~\eqref{des:asc}.

\begin{theorem}[Foata~1977]\label{foata}
For each $\pi\in\S_n$ let $\ides(\pi):=\des(\pi^{-1})$ be the number of {\bf\em i}nverse {\bf\em des}cents of $\pi$. Then,
\begin{equation}\label{dist:asc}
\sum_{\pi\in\S_n}s^{\ides(\pi)}t^{\DES(\pi)}=\sum_{e\in\I_n}s^{\dist(e)}t^{\ASC(e)}.
\end{equation}
\end{theorem}
 
Recently, this result was rediscovered by Visontai~\cite{vi} during his study of a conjecture of Gessel about the 2-D $\gamma$-positivity of the {\em double Eulerian polynomial} $\sum_{\pi\in\S_n}s^{\ides(\pi)}t^{\des(\pi)}$.
Note that a new proof of~\eqref{dist:asc} essentially different from Foata's was provided by Aas~\cite{aas}.
 
Next we define the patterns in words. For two words $W=w_1w_2\cdots w_n$ and $P=p_1p_2\cdots p_k$  on $\N=\{0,1,2,\dots\}$, we say that {\em$W$ contains the pattern $P$} if there exist some indices $i_1<i_2<\cdots<i_k$ such that the subword $W'=w_{i_1}w_{i_2}\cdots w_{i_k}$ of $W$ is order isomorphic to $P$. Otherwise, $W$ is said to {\em avoid the pattern $P$}. For example, the word $W=32421$ contains the pattern $231$, because the subword $w_2w_3w_5=241$ of $W$ has the same relative order as $231$.
However, $W$ is $101$-avoiding. For a set $\mathcal{W}$ of words,
let $\mathcal{W}(P_1,\ldots,P_r)$ denote the set of words in $\mathcal{W}$ avoiding
patterns $P_1,\ldots,P_r$. Permutations and inversion sequences are viewed as words
in this paper so that patterns are well defined on them. 
 
Patterns in permutations and words have been extensively studied in the literature (see for instance~\cite{ki}). One classical result attributed to MacMahon and Knuth is that, for each pattern $\sigma\in\S_3$, $\S_n(\sigma)$ is enumerated by the {\em Catalan number} $C_n=\frac{1}{n+1}{2n\choose n}$. West~\cite{we} studied patterns of length 4 and showed that $(2413,3142)$-avoiding permutations, also known as {\em separable permutations}, are counted by the {\em (large) Schr\"oder numbers} 
$$
\{S_n\}_{n\geq1}=\{1, 2, 6, 22, 90, 394, 1860,\ldots\}_{n\geq1}.
$$
Kremer~\cite{kre} characterized that, among all permutation classes avoiding two patterns of length 4, there are only ten nonequivalent classes which have cardinality $S_n$, including separable permutations and $(2413,4213)$-avoiding permutations. Recently, motivated again by Gessel's $\gamma$-positivity conjecture above, Fu et al.~\cite{flz} proved that the descent polynomial $S_n(t):=\sum_{\pi\in\S_n(2413,3142)}t^{\des(\pi)}$ is $\gamma$-positive, which implies that the sequence of coefficients of $S_n(t)$ is {\em palindromic} and {\em unimodal}. 
Very recently, Corteel et al.~\cite{cor} proved that the Schr\"oder numbers also count $021$-avoiding inversion sequences. Moreover, they showed that the ascent polynomial $\sum_{e\in\I_n(021)}t^{\asc(e)}$ is palindromic via a connection with {\em rooted binary trees} with nodes colored by black or white. In fact, these two seemingly different classes of palindromic polynomials are the same, that is
\begin{equation}\label{eq:sep}
\sum_{\pi\in\S_n(2413,3142)}t^{\des(\pi)}=\sum_{e\in\I_n(021)}t^{\asc(e)}.
\end{equation}
It is this result that inspires us to find the following set-valued sextuple equidistribution, which is an extension of a restricted version of Theorem~\ref{foata}. 

Let us introduce the set-valued statistics, for each $\pi=\pi_1\pi_2\dots\pi_n\in\S_n$:
\begin{itemize}
\item $\VID(\pi):=\{2\leq i\leq n:\pi_i+1 \,\,\text{appears to the left of}\,\, \pi_i\}$, the {\bf v}alues of {\bf i}nverse {\bf d}escents of $\pi$; 
\item $\LMA(\pi):=\{i\in[n]:\pi_i>\pi_j\,\, \text{for all $1\leq j<i$}\}$, the positions of {\bf l}eft-to-right {\bf ma}xima of $\pi$;
\item $\LMI(\pi):=\{i\in[n]:\pi_i<\pi_j\,\, \text{for all $1\leq j<i$}\}$, the positions of {\bf l}eft-to-right {\bf mi}nima of $\pi$;
\item $\RMA(\pi):=\{i\in[n]:\pi_i>\pi_j\,\, \text{for all $j>i$}\}$, the positions of {\bf r}ight-to-left {\bf ma}xima of $\pi$;
\item $\RMI(\pi):=\{i\in[n]:\pi_i<\pi_j\,\, \text{for all $j>i$}\}$, the positions of {\bf r}ight-to-left {\bf mi}nima of~$\pi$;
\end{itemize}
and for each $e=e_1 e_2\dots e_n\in\I_n$:
\begin{itemize}
\item $\DIST(e):=\{2\leq i\leq n: e_i\neq0\,\,\text{and $e_i\neq e_j$ for all $j>i$}\}$, the positions of the last occurrence of {\bf dist}inct positive entries of $e$; 
\item $\ZERO(e):=\{i\in[n]: e_i=0\}$, the positions of {\bf zero}s in $e$;
\item $\EMA(e):=\{i\in[n]: e_i=i-1\}$, the positions of the {\bf e}ntries of $e$ that achieve {\bf ma}ximum;
\item $\RMI(e):=\{i\in[n]: e_i<e_j\,\, \text{for all $ j>i$}\}$, the positions of {\bf r}ight-to-left {\bf mi}nima of~$e$.
\end{itemize}
\begin{example}
Take $\pi=5\,3\,6\,8\,7\,4\,9\,1\,11\,12\,10\,2$ in $\S_{12}$ and
$e=(0,1,0,0,1,3,0,7,0,0,7,10)$ in $\I_{12}$.
We have $\DES(\pi)=\{1,4,5,7,10,11\}=\ASC(e)$, $\VID(\pi)=\{5,6,11,12\}=\DIST(e)$,
$\LMA(\pi)=\{1,3,4,7,9,10\}=\ZERO(e)$, $\LMI(\pi)=\{1,2,8\}=\EMA(e)$,
$\RMA(\pi)=\{10,11,12\}=\RMI(e)$ and $\RMI(\pi)=\{8,12\}$.
\end{example}

\begin{theorem}\label{thm:sex} 
There exists a bijection $\Psi:\I_n(021) \rightarrow \S_n(2413,4213)$ such that
$$
(\DIST,\ASC,\ZERO,\EMA,\RMI,\EXPO)e=(\VID,\DES,\LMA,\LMI,\RMA,\RMI)\Psi(e)
$$
for each $e\in\I_n(021)$, where $\EXPO(e)$ is the {\bf\em expo}sed positions of $e$ introduced in Definition~\ref{outline}.
\end{theorem}

Since $|\VID(\pi)|=\ides(\pi)$ for any $\pi\in\S_n$, the above bijection implies the following restricted version of Theorem~\ref{foata}:
\begin{equation}\label{restrict}
\sum_{e\in\I_n(021)}s^{\dist(e)}t^{\ASC(e)}=\sum_{\pi\in\S_n(2413,4213)}s^{\ides(\pi)}t^{\DES(\pi)}.
\end{equation} 
Note that neither Foata's original bijection nor Aas' approach can be applied to prove this restricted
version. This equidistribution has two interesting applications:
\begin{itemize}
\item The calculation of the double Eulerian distribution $(\des,\ides)$ on $\S_n(2413,4213)$;
\item An interpretation of the $\gamma$-coefficients of $S_n(t)$ in terms of $021$-avoiding inversion
sequences via the {\em Foata--Strehl group action}~\cite{fsh} (see also~\cite{lz,pe}). 
\end{itemize}
Moreover, inspired by equidistributions~\eqref{eq:sep} and~\eqref{restrict}, we also characterize all
permutation classes that avoid two patterns of length $4$ whose descent distribution give the
palindromic polynomials $S_n(t)$.

The rest of this paper is organized as follows. In Section~\ref{sec:psi}, by introducing the outline of
an inversion sequence, we construct the algorithm $\Psi$. We show that $\Psi$ is bijective by giving
its inverse explicitly in Section~\ref{sec:inverse}. Two applications of $\Psi$ are presented in
Section~\ref{sec:app}. In Section~\ref{sec:wilf}, we investigate some $\des$-Wilf equivalences for
Schr\"oder classes avoiding two length $4$ patterns. Finally, we conclude the paper with some remarks. 
 
%%%%%%%%%%%%%%%%%%%%%%%%%%%
\section{The algorithm $\Psi$}\label{sec:psi}
Recall that a {\em Dyck path} of length $n$ is a lattice path in $\N^2$ from $(0,0)$ to
$(n,n)$ using the {\em east step} $(1,0)$ and the {\em north step} $(0,1)$, which does
not pass above the line $y=x$. The {\em height of an east step} in a Dyck path is
the number of north steps before this east step. It is clear that a Dyck path can be
represented as $d_1d_2\ldots d_n$, where $d_i$ is the height of its $i$-th east step.
For our purpose, we color each east step of a Dyck path by {\em black} or {\em red}.
We call such a Dyck path a {\em two-colored Dyck path}. If the $i$-th east step has
height $k$ and color red, then we write $d_i=\bar{k}$.
For example, the two-colored Dyck path in Figure~\ref{fig:outline} (right side)
can be written as $\bar{0}1\bar{1}12\bar{2}4$.

\begin{figure}%[ht]
\begin{center}
\begin{tikzpicture}[scale=.6]
\draw[step=1,color=gray](0,1) grid (7,8);
\draw [very thick](1,2)--(2,2)(3,2)--(4,2)(4,3)--(5,3)(6,5)--(7,5);
\draw [very thick,color=red](0,1)--(1,1)(2,1)--(3,1)(5,1)--(6,1);
\draw [color=blue](0,1)--(7,8);
\draw(0.5,0.5) node{$0$};
\draw(1.5,0.5) node{$1$};
\draw(2.5,0.5) node{$0$};
\draw(3.5,0.5) node{$1$};
\draw(4.5,0.5) node{$2$};
\draw(5.5,0.5) node{$0$};
\draw(6.5,0.5) node{$4$};
\draw(9,5) node{$\mapsto$};
\draw(8.95,5.6) node{$d$};

\draw[step=1,color=gray](11,1) grid (18,8); 
\draw [color=blue](11,1)--(18,8);
\draw [very thick,color=red](11,1)--(12,1);
\draw [very thick](12,1)--(12,2)--(13,2);
\draw [very thick,color=red](13,2)--(14,2);
\draw [very thick](14,2)--(15,2)--(15,3)--(16,3);
\draw [very thick,color=red](16,3)--(17,3);
\draw [very thick](17,3)--(17,5)--(18,5)--(18,8);
\end{tikzpicture}
\end{center}
\caption{The outline of an inversion sequence}
\label{fig:outline}
\end{figure}
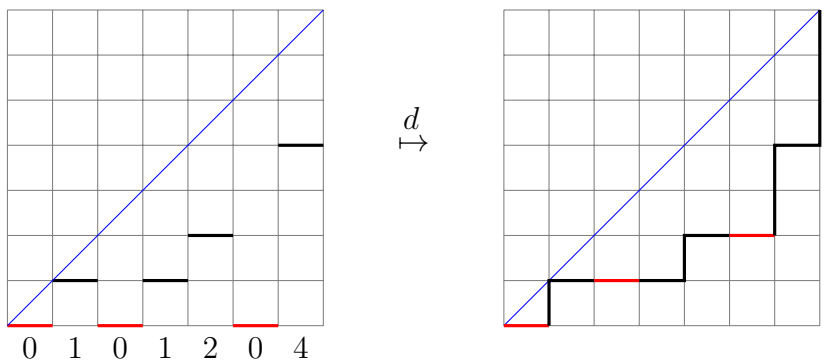

Note that an inversion sequence avoids $021$ if and only if its positive entries are weakly increasing.
This inspires the following geometric representation which is crucial in our construction of $\Psi$.

\begin{Def}[The {\em outline} of an inversion sequence]\label{outline} For each $e\in\I_n(021)$,
we associate it with a two-colored Dyck path $d(e)=d_1d_2\ldots d_n$, where the red east steps, indicated
by an overbar, correspond to the positions of zero entries of $e$ like this: 
$$
d_{i}=
\begin{cases}
\,\,e_i\quad&\text{if $e_i\neq0$},\\
\,\,\bar{k}\quad&\text{if $e_i=0$ and $k=\max\{e_1,\ldots,e_i\}$}.
\end{cases}
$$
For example, if $e=(0,1,0,1,2,0,4)\in\I_7(021)$, then $d(e)$ is the two-colored Dyck path in
Figure~\ref{fig:outline}. The two-colored Dyck path $d(e)$ is called the {\em outline} of $e$.
We introduce the {\em {\bf \em expo}sed positions} of $e=e_1e_2\dots e_n$ as 
$$
\EXPO(e):=\{i: i\notin\mathcal{C}(e) \text{ and $i-d_i<j-d_j$ for all $j>i$ }\},
$$
where $\bar{k}$ is just $k$ as a number and
$$
\mathcal{C}(e)=\{i:\text{$e_i=0 $ and there is $(a,b)$, $a<i<b$, satisfying $e_a=e_b\neq0$}\}.
$$
Continuing with our example, we have $\EXPO(e)=\{2,7\}$.
\end{Def}

It is obvious that two different $021$-avoiding inversion sequences can not have
the same outline, i.e., the geometric representation $d$ is an injection
from $\I_n(021)$ to two-colored Dyck paths of length $n$. Now we are ready to
present our algorithm $\Psi$ using this geometric representation. 

\vskip 0.1in

{\bf The algorithm $\Psi$}. Given $e\in\I_n(021)$, draw its outline
$d(e)=d_1d_2\ldots d_n$ on the $\N^2$ plane. For each $i\in[n]$, let $E_i$ denote
the $i$-th east step of $d(e)$. By {\em drawing a line} on $d(e)$, we mean drawing
a longest line segment which is parallel to the diagonal $y=x$ with two end points
in $\N^2$ and is contained in the region bounded by $d(e)$ and the diagonal.
For example, in the following
\begin{center}
\begin{tikzpicture}[scale=.5]
%%%%%good lines
\draw[step=1,color=gray] (0,0) grid (7,7); 
\draw [color=blue](0,0)--(7,7);
\draw [very thick,color=red](0,0)--(1,0);
\draw [very thick](1,0)--(1,1)--(2,1);
\draw [very thick,color=red](2,1)--(3,1);
\draw [very thick](3,1)--(4,1)--(4,2)--(5,2);
\draw [very thick,color=red](5,2)--(6,2);
\draw [very thick](6,2)--(6,4)--(7,4)--(7,7);
\draw [color=blue](2,1)--(7,6);
\draw [color=blue](3,1)--(7,5);
\draw [color=blue](5,2)--(6,3);
\draw(7.4,7) node{$l_1$};
\draw(7.4,6) node{$l_2$};
\draw(7.4,5) node{$l_3$};
\draw(6.4,3) node{$l_4$};
%%%%%bad lines
\draw[step=1,color=gray] (12,0) grid (19,7); 
\draw [color=blue](12,0)--(19,7);
\draw [very thick,color=red](12,0)--(13,0);
\draw [very thick](13,0)--(13,1)--(14,1);
\draw [very thick,color=red](14,1)--(15,1);
\draw [very thick](15,1)--(16,1)--(16,2)--(17,2);
\draw [very thick,color=red](17,2)--(18,2);
\draw [very thick](18,2)--(18,4)--(19,4)--(19,7);
\draw [dashed,color=green](13,0)--(19,6);
\draw [dashed,color=green](15,1)--(18,4);
\end{tikzpicture}
\end{center}
all the blue lines are good for our purpose, but the two green (dashed) lines are not,
one of which is not contained in the region bounded by $d(e)$ and the diagonal and
the other is not of a maximal length.
We say a line {\em touches an east step} of $d(e)$ if the line meets the east step at
the initial, equivalently leftmost, point of the step. Moreover, a line is said to
begin at the leftmost east step that it touches. We also say that a line is to the left
or to the right of another line according to the location of the east steps they begin at.
For instance, the line $l_3$ in the graph above,
where the lines are $l_1,l_2,l_3,l_4$ from left to right, begins at $E_4$ which touches the east steps $E_4,E_5$ and $E_7$ of $d(e)$.

In our algorithm, we will draw all the lines on $d(e)$ in some specified order and
label the east steps touched by the lines successively with integers in $[n]$
according to two particular rules that we describe below:
\begin{itemize}
\item[(r-a)] A red east step can be labeled only if all the red east steps to its left are already labeled.
\item[(r-b)] Whenever an east step $E_j$ is labeled after an east step $E_i$ with $j<i$, we must finish labeling
all the other east steps left to $E_i$ before we can label any east step to the right of $E_i$.
\end{itemize}
When labeling an east step does not violate the above rules, the step is said to be~\emph{labelable}.

\begin{figure}[h!]
\scalebox{.65}{\includegraphics{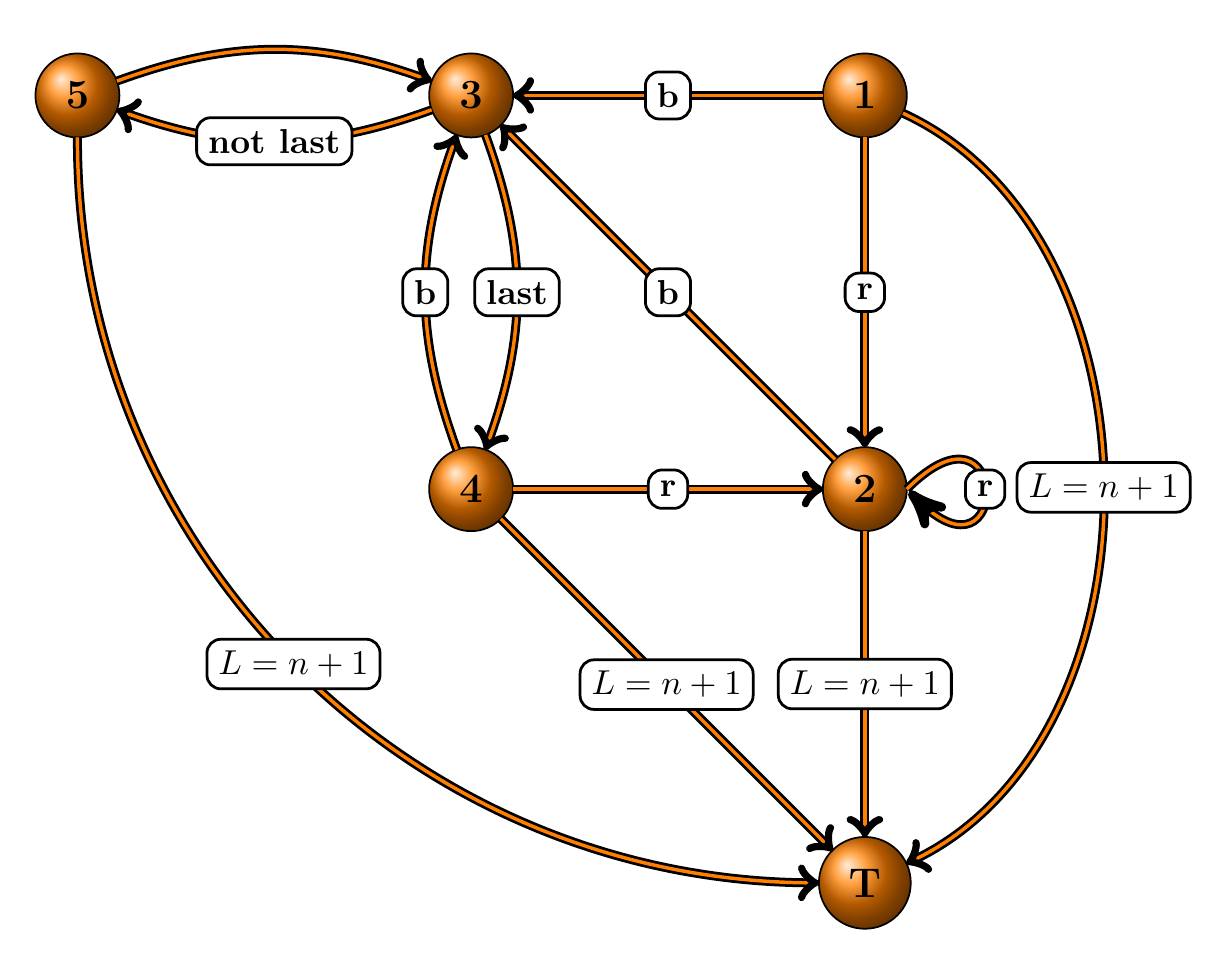}}
\caption{The flowchart of algorithm $\Psi$: {\bf r} (resp.~{\bf b}) means $E_P$ is
a red (resp.~black) east step, {\bf last} means $E_P$ is the last occurrence of
a black east step with the same height and {\bf T} means the algorithm terminates.
\label{flowchart}}     
\end{figure}

We now present step-by-step instructions to obtain $\Psi(e)\in\S_n(2413,4213)$ from
the outline $d(e)$ of $e\in\I_n(021)$. The labeling algorithm, where temporary
variables $L,H,P$ correspond to words \emph{label, height, position}, works as follows (see also the flowchart for $\Psi$ in Figure~\ref{flowchart}):
\begin{enumerate}
\item (Start) $L\leftarrow1$ (This means that $1$ is assigned to $L$);
draw the diagonal (line) $y=x$ on $d(e)$ and label the highest east step touched by the diagonal,
say $E_k$, with $L$; $L\leftarrow L+1$, $H\leftarrow d_k$, $P\leftarrow k$;
go to (2), if $E_P$ is a red east step (i.e.~$k=1$), otherwise go to~(3);
\item draw the leftmost new line that touches at least one unlabeled black east step or a \emph{labelable} red east step; label the highest east step touched by this new line, say $E_k$, with $L$; $L\leftarrow L+1$, $H\leftarrow d_k$, $P\leftarrow k$; go to (2), if $E_P$ is a red east step, otherwise go to~(3).
\item go to (5), if there is a black east step $E_j$ with $j>P$ and height $d_j=H$, otherwise go to (4);
\item move from $E_P$ along the two-colored Dyck path $d(e)$ to the left and along the lines that were
already drawn to the southwest until we arrive at the first unlabeled east step that is a black step
or a \emph{labelable} red step, say $E_k$; label $E_k$ with $L$;
$L\leftarrow L+1$, $H\leftarrow d_k$, $P\leftarrow k$; go to (2), if $E_P$ is a red east step, otherwise go to~(3); 
\item draw the leftmost line beginning at an east step right to $E_P$ which touches at least
one black east step; label the highest east step touched by this new line, say $E_k$, with $L$;
$L\leftarrow L+1$, $H\leftarrow d_k$, $P\leftarrow k$; go to (3);
\end{enumerate}
Finally, this algorithm stops when $L=n+1$ and $\Psi(e)$ is the resulting permutation $\pi_1\pi_2\ldots\pi_n$, where $\pi_i$ is the label of $E_i$. Clearly our algorithm $\Psi$ obeys the two rules (r-a) and (r-b). Therefore, the final permutation $\Psi(e)$ avoids both $2413$ and $4213$ in view of rule (r-b), and so $\Psi$ is well defined. 

\begin{figure}%[ht]
\begin{center}
\begin{tikzpicture}[scale=.5]
\draw[step=1,color=gray] (0,1) grid (12,13); 
\draw [color=blue](0,1)--(12,13);
\draw [very thick,color=red](0,1)--(1,1);
\draw [very thick](1,1)--(1,2)--(2,2);
\draw [very thick,color=red](2,2)--(3,2)--(4,2);
\draw [very thick](4,2)--(5,2)--(5,4)--(6,4);
\draw [very thick,color=red](6,4)--(7,4);
\draw [very thick](7,4)--(7,8)--(8,8);
\draw [very thick,color=red](8,8)--(10,8);
\draw [very thick](10,8)--(11,8)--(11,11)--(12,11)--(12,13);
\draw [color=blue](8,8)--(12,12);
\draw [color=blue](2,2)--(7,7);
\draw [color=blue](3,2)--(7,6);
\draw [color=blue](4,2)--(5,3);
\draw [color=blue](6,4)--(7,5);
\draw [color=blue](10,8)--(11,9);
\draw [color=blue](9,8)--(11,10);

\draw(12.4,13) node{$l_1$};
\draw(12.4,12) node{$l_2$};
\draw(7.4,7) node{$l_4$};
\draw(7.4,6) node{$l_3$};
\draw(5.4,3) node{$l_5$};
\draw(7.4,5) node{$l_6$};
\draw(11.4,9) node{$l_7$};
\draw(11.4,10) node{$l_8$};

\draw(0.5,0.5) node{$5$};
\draw(1.5,0.5) node{$3$};
\draw(2.5,0.5) node{$6$};
\draw(3.5,0.5) node{$8$};
\draw(4.5,0.5) node{$7$};
\draw(5.5,0.5) node{$4$};
\draw(6.5,0.5) node{$9$};
\draw(7.5,0.5) node{$1$};
\draw(8.5,0.5) node{$11$};
\draw(9.5,0.5) node{$12$};
\draw(10.5,0.5) node{$10$};
\draw(11.5,0.5) node{$2$};

\draw[step=1,color=gray] (-16,1) grid (-4,13); 
\draw(-2,7) node{$\mapsto$};
\draw(-2,7.7) node{$\Psi$};
\draw [color=blue](-16,1)--(-4,13);

\draw [very thick](-15,2)--(-14,2)(-12,2)--(-11,2)(-11,4)--(-10,4)(-9,8)--(-8,8)(-6,8)--(-5,8)(-5,11)--(-4,11);

\draw [very thick,color=red](-16,1)--(-15,1)(-14,1)--(-12,1)(-10,1)--(-9,1)(-8,1)--(-6,1);

\draw(-15.5,0.5) node{$0$};
\draw(-14.5,0.5) node{$1$};
\draw(-13.5,0.5) node{$0$};
\draw(-12.5,0.5) node{$0$};
\draw(-11.5,0.5) node{$1$};
\draw(-10.5,0.5) node{$3$};
\draw(-9.5,0.5) node{$0$};
\draw(-8.5,0.5) node{$7$};
\draw(-7.5,0.5) node{$0$};
\draw(-6.5,0.5) node{$0$};
\draw(-5.5,0.5) node{$7$};
\draw(-4.5,0.5) node{$10$};
\end{tikzpicture}
\end{center}
\caption{An example of the algorithm $\Psi$}
\label{fig:Psi}
\end{figure}
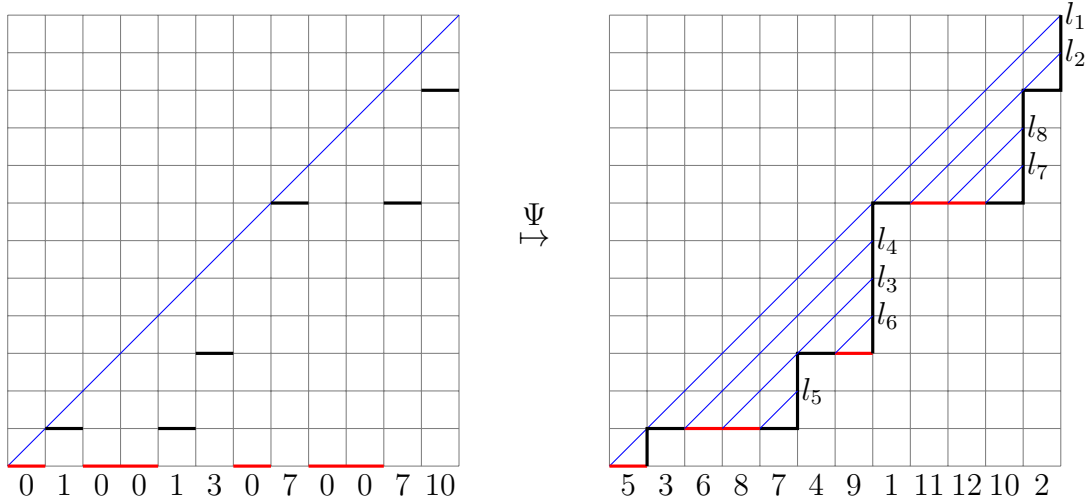

\begin{example}Let $e=(0,1,0,0,1,3,0,7,0,0,7,10)\in\I_{12}(021)$. In Figure~\ref{fig:Psi}, draw the outline $d(e)$ and apply the algorithm $\Psi$: 
\begin{itemize}
\item do (1), draw the line $l_1$ on $d(e)$ and label $E_8$ with $1$; 
\item do (5), draw the line $l_2$ and label $E_{12}$ with $2$; 
\item do (4), move from $E_{12}$ along the line $l_2$ to $E_9$ which violates (r-a), then move to $E_8$ which was already labeled, keep moving along the line $l_1$ until we arrive at $E_2$ and label it with $3$; 
\item do (5), draw $l_3$ and label $E_6$ with $4$; 
\item do (4), move along $l_3$ to $E_4$ which violates (r-a), then move to $E_3$ which still violates (r-a), further move to $E_2$ which was already labeled, finally, along $l_1$ we arrive at $E_1$ and label it with $5$; 
\item do (2), draw the line $l_4$ and label $E_3$ with $6$; 
\item do (2), draw the line $l_5$ and label $E_5$ with $7$; 
\item do (4), move from $E_5$ to $E_4$ which obeys (r-a) and is labeled with $8$; 
\item do (2), draw the line $l_6$ and label $E_7$ with $9$;
\item do (2), draw the line $l_7$ (although $l_8$ is left to $l_7$, it can not be drawn at this moment, because the only east step it touches is $E_{10}$, which is red and violates (r-a)) and label $E_{11}$ with $10$; 
\item do (4), move from $E_{11}$ to $E_{9}$ and label it with $11$; 
\item do (2), draw $l_8$ and label $E_{10}$ with $12$. 
\end{itemize}
The resulting permutation $\Psi(\pi)$ is $5\,3\,6\,8\,7\,4\,9\,1\,11\,12\,10\,2\in\S_{12}(2413,4213)$.
\end{example}

%%%%%%%%%%%%%%%%%%%%%%%%%%%
\section{The inverse algorithm $\Psi^{-1}$}
\label{sec:inverse}
We prove in this section that $\Psi$ is a bijection by constructing its inverse. First we show the following required properties of $\Psi$.

\begin{lemma}\label{lem:lrm}
Let $e\in\I_n(021)$ and $\pi=\Psi(e)\in\S_n(2413,4213)$. Then, for each $i\in[n]$:
\begin{enumerate}
\item $i$ is an ascent of $e$ if and only if $i$ is a descent of $\pi$;
\item $e_i\neq0$ does not occur in $e_{i+1}e_{i+2}\cdots e_n$ if and only if $i\in\VID(\pi)$;
\item $e_i=0$ if and only if $\pi_i$ is a left-to-right maximum of $\pi$;
\item $e_i=i-1$ if and only if $\pi_i$ is a left-to-right minimum of $\pi$;
\item $i$ is the position of a right-to-left minimum of $e$ if and only if
$\pi_i$ is a right-to-left maximum of $\pi$;
\item $i\in\EXPO(e)$ if and only if $\pi_i$ is a right-to-left minimum of $\pi$.
\end{enumerate}
\end{lemma}
\begin{proof} Let $d(e)=d_1d_2\ldots d_n$ be the outline of $e$ and $E_i$ the $i$-th step of $d(e)$ for each $i\in[n]$. 
We establish the above assertions one by one.
\begin{enumerate}[(1)]
\item If $i\in\ASC(e)$, then we need to consider three different cases:
\begin{itemize}
\item $e_i=0,e_{i+1}>0$ and $d_i=e_{i+1}$;
\item $e_i=0,e_{i+1}>0$ and $d_i<e_{i+1}$;
\item $0<e_i<e_{i+1}$.
\end{itemize}
One can check easily that in either case, step $E_i$ is labeled after $E_{i+1}$ by algorithm $\Psi$, i.e. $i\in\DES(\pi)$.
On the other hand, if $i\notin\ASC(e)$, then again one can check routinely that in all cases below
\begin{itemize}
\item $e_{i}=e_{i+1}=0$,
\item $e_i\neq0, e_{i+1}=0$ or
\item $e_i=e_{i+1}\neq0$,
\end{itemize}
step $E_i$ is labeled before $E_{i+1}$ by algorithm $\Psi$, i.e. $i\notin\DES(\pi)$.
\item Clearly, $i\in\DIST(e)$ if and only if procedure (4) is called from procedure (2) with $P=i$ in algorithm $\Psi$
for $d(e)$, and procedure (4) is called with $P=i$ and $L=j$ if and only if $j-1$ is the label for $E_{i}$ and $j$ becomes the label of an east step to the left of $E_i$.
\item
Let $e_i=0$. We may assume that $i<n$ and $e_{i+1}>0$. Let $j$ be the largest index such that
$e_k>0$ for all $k$, $i<k\leq j$. Since $e$ is $021$-avoiding, the entries $e_{i+1},e_{i+2},\ldots,e_j$ are all positive and weakly increasing. For any pair $(a,b)$ with $i\leq a<b\leq j$ satisfying 
$$
e_a<e_{a+1}=e_{a+2}=\cdots=e_b\neq e_{b+1},
$$
we claim that $\pi_b<\pi_a$ and $\pi_{a+1}<\pi_{a+2}<\cdots<\pi_{b}$. The latter statement is clear from the algorithm $\Psi$ and $\pi_b<\pi_a$ follows from the the fact that $E_a$ is labeled after $E_{a+1}$ and the observation that all east steps before $E_{a+1}$ which are labeled after $E_{a+1}$ must be labeled after $E_{a+2},\ldots,E_{b}$. 
It then follows from the above discussion that $\pi_i>\pi_{i'}$ for any $i<i'\leq j$. Therefore, (3) is true in view of rule (r-a) of $\Psi$.

\item In the algorithm, $e_i=i-1$ if and only if step $E_i$ is touched by the diagonal line $y=x$, equivalently, $\pi_i$ is a left-to-right minimum of $\pi$.
\item Let $i$ be the largest such that $e_i=0$. Applying the claim in (3) gives (5).
\item If $i\in\EXPO(e)$, then neither the line $l$ that touches $E_i$ nor the lines to the left of $l$ will touch an east step to the right of $E_i$. This shows $\pi_i$ is a right-to-left minimum of $\pi$. 

Conversely, if $i\notin\EXPO(e)$, then we need to distinguish two cases. Clearly, if there exists $j>i$ such that $d_i\leq d_j+i-j$, then either the line $l$ that touches $E_i$ or a line to the left of $l$ will touch a black east step to the right of and higher than $E_i$, which forces $i\notin\RMI(\pi)$. Otherwise, $i\in\mathcal{C}(e)$, then either the line $l$ that touches $E_i$ will touch a higher black east step or a line to the right of $l$ will be drawn before $l$, which again forces $i\notin\RMI(\pi)$. 
\end{enumerate}
\end{proof}

The following definition is important in constructing the algorithm $\Psi^{-1}$.
\begin{Def}[Big jump]
For a permutation $\pi\in\S_n$ with 
$$\LMA(\pi)=\{1=i_1<i_2<\cdots <i_k\},$$
an index $i_j\in\LMA(\pi)$ ($2\leq j\leq k$) is called a {\em big jump} of $\pi$ if
$\pi_{i_j}-\pi_{i_{j-1}}>1$. Denote by $\BJP(\pi)$ the set of all {\bf b}ig {\bf j}um{\bf p}s of $\pi$.
For example, if $\pi=5\,3\,6\,8\,7\,4\,9\,1\,11\,12\,10\,2\in\S_{12}$, then
$\LMA(\pi)=\{1,3,4,7,9,10\}$ and $\BJP(\pi)=\{4,9\}$.
\end{Def}

{\bf The algorithm $\Psi^{-1}$}. As before, all lines we will draw below are parallel to the diagonal $y=x$.
For a line $l: y=x-k$ ($0\leq k\leq n-1$) on $\N^2$, when we say drawing the $j$-th ($1\leq j\leq n$)
east step that touches the line $l$, we mean drawing the east step from $(j-1,j-1-k)$ to $(j,j-1-k)$.
Normally, for $i<j$ let $(i,j]:=\{i+1,i+2,\ldots,j\}$. 

Given a permutation $\pi=\Psi(e)$ for some $e\in\I_n(021)$, we can recover $d(e)=d_1\cdots d_n$ as follows.
Note that the $i$-th east step of $d(e)$ is red if and only if $i\in\LMA(\pi)$, by (3) of Lemma~\ref{lem:lrm}. We are going to determine  the height $d_{\pi^{-1}(1)}, d_{\pi^{-1}(2)},\ldots,d_{\pi^{-1}(n)}$ successively using the following algorithm: 
\begin{enumerate}[(I)]
\item (Start) $L\leftarrow1$; let $j=\pi^{-1}(L)$, draw the diagonal (line) $y=x$
(on $\N^2$) and the $j$-th east step (red or black) that touches this diagonal;
%2016.12.19 dskim
%$L\leftarrow L+1$; go to (V), if $j\in\VID(\pi)$, otherwise go to (III);
$L\leftarrow L+1$; go to (II), if $j\in\VID(\pi)$, otherwise go to (III);
\item let $j=\pi^{-1}(L)$ and draw the $j$-th east step (red or black) with $d_j$ smallest possible
such that it touches a line already drawn; $L\leftarrow L+1$;
go to (II), if $j\in\VID(\pi)$, otherwise go to (III);
\item go to (V), if $\pi^{-1}(L)\in\LMA(\pi)$, otherwise go to (IV);
\item let $i=\pi^{-1}(L-1)$, $j=\pi^{-1}(L)$ and 
$d=\max\{d_k: \text{$i\leq k<j$ and $\pi_k<L$}\}$, ignoring undetermined $d_k$'s;
let $m$ be the smallest integer satisfying that
$$
m\in(i,j]\setminus\left(\LMA(\pi)\setminus\BJP(\pi)\right)
\text{ and there is no line from $(m-1,d)$;}
$$
draw the line starting in $(m-1,d)$ and draw the $j$-th east step (black) that touches this line;
$L\leftarrow L+1$; go to (II), if $j\in\VID(\pi)$, otherwise go to (III);
\item let $j=\pi^{-1}(L)$ and draw the $j$-th east step (red) with $d_j=\max\{d_1,d_2,\ldots,d_{j-1}\}$;
$L\leftarrow L+1$; go to (III).
\end{enumerate}

The above algorithm ends when $L=n+1$ and we recover the outline $d(e)$ of $e$. For example, if we apply
the algorithm $\Psi^{-1}$ to $\pi=5\,3\,6\,8\,7\,4\,9\,1\,11\,12\,10\,2\in\S_{12}(2413,4213)$,
then we will successively draw the lines $l_1,l_2,l_3,l_5,l_7$ and recover the corresponding outline
(see Figure~\ref{fig:Psi}). Note that the lines touching only a red east step (e.g.~lines $l_4,l_6,l_8$
in Figure~\ref{fig:Psi}), which are unnecessary, will not be drawn during the procedure.
These lines are exactly the lines beginning with a red east step which is not a big jump step
(the $i$-th step is called a big jump step if $i\in\BJP(\pi)$). At each big jump step there must be
a line drawn in procedure (IV) of $\Psi^{-1}$ (with $m\in\mathcal{B}_{i,j}$), which corresponds to
a line drawn in procedure (2) of $\Psi$ that begins at a red east step and touches at least one black
east step. By the above discussion, one can check without difficulty that $\Psi^{-1}(\Psi(e))=e$ for
each $e\in\I_n(021)$, i.e.~$\Psi$ is injective. Therefore, $\Psi$ is a bijection
(due to cardinality reason) with all the desired properties as shown in Lemma~\ref{lem:lrm},
which completes the proof of Theorem~\ref{thm:sex}.

%%%%%%%%%%%%%%%%%%%%%%%%%%%
\section{Two applications}
\label{sec:app}
\subsection{Double Eulerian distribution on $\S_n(2413,4213)$}

\label{sec:distri}
One advantage of our bijection $\Psi$ is the calculation of the double Eulerian distribution
$(\des,\ides)$ on $\S_n(2413,4213)$ via the natural structure of two-colored Dyck paths. 
\begin{Def}
 For a two-colored Dyck path $D$, we introduce the following four statistics:
\begin{itemize}
\item $\turn(D)$, the number of {\em{\bf\em turn}s} of $D$ minus $1$, where a turn is an east step
that is followed immediately by a north step;
\item $\segm(D)$, the number of {\em black {\bf\em segment}s} of $D$, where a segment is a maximal
string of consecutive black east steps with the same height;
\item $\Red(D)$, the number of {\em{\bf\em red} east steps} of $D$;
\item $\return(D)$, the number of {\em{\bf\em return}s} of $D$, i.e. the number of times that $D$
touches the diagonal $y=x$ after the first east step.
\end{itemize}
For example, for the two-colored Dyck path
$D=\bar{0}\bar{0}2\bar{2}2\bar{2}57\bar{7}\bar{7}77$,
we have $\turn(D)=3$, $\segm(D)=5$, $\Red(D)=6$ and $\return(D)=3$. 
\end{Def}

It is clear from the definition of outline that a two-colored Dyck path is
an outline of an inversion sequence if and only if
\begin{itemize}
\item[(c-1)] all the east steps of height $0$ are colored red, and
\item [(c-2)] the first east step of each positive height is colored black.
\end{itemize}
Denote by $\mathcal{A}_n$ the set of all two-colored Dyck paths of length $n$ that satisfy
the above two conditions.
The following Lemma~\ref{lem:outl} is obvious from the definition of the injection $d$. 

\begin{lemma}\label{lem:outl}
The geometric representation of inversion sequences $d:\I_n(021)\rightarrow\mathcal{A}_n$
is a bijection such that for each $e\in\I_n(021)$,
\begin{equation}
(\dist,\asc,\zero,\ema)(e)=(\turn,\segm,\Red,\return)(d(e)).
\end{equation}
\end{lemma}

Lemma~\ref{lem:outl} and Theorem~\ref{thm:sex} together give the following interpretations 
for the four variable extension of $S_n(t)$:
\begin{align*}
S_n(s,t,u,v)&:=\sum_{D\in\mathcal{A}_n}s^{\turn(D)}t^{\segm(D)}u^{\Red(D)}v^{\return(D)}\\
&\,\,=\sum_{e\in\I_n(021)}s^{\dist(e)}t^{\asc(e)}u^{\zero(e)}v^{\ema(e)}\\
&\,\,=\sum_{\pi\in\S_n(2413,4213)}s^{\ides(\pi)}t^{\des(\pi)}u^{\lma(\pi)}v^{\lmi(\pi)}.
\end{align*}
In order to compute the generating function for $S_n(s,t,u,v)$, we further need to introduce two related polynomials:
\begin{align*}
B_n(s,t,u,v)&:=\sum_{D\in\mathcal{B}_n}s^{\turn(D)}t^{\segm(D)}u^{\Red(D)}v^{\return(D)}\\
\text{and}\quad R_n(s,t,u,v)&:=\sum_{D\in\mathcal{R}_n}s^{\turn(D)}t^{\segm(D)}u^{\Red(D)}v^{\return(D)},
\end{align*}
where $\mathcal{B}_n$ (resp.~$\mathcal{R}_n$) is the set of all two-colored Dyck paths of length $n$
satisfying condition~(c-2) and whose first step is black (resp.~red). We have the following system of
algebraic equations regarding the three generating functions $S(s,t,u,v;z)$, $B(s,t,u,v;z)$ and $R(s,t,u,v;z)$: 
\begin{align*}
S(s,t,u,v;z)&:=\sum_{n\geq1}S_n(s,t,u,v)z^n=uvz+(stuv^2+u^2v)z^2+\cdots\\
B(s,t,u,v;z)&:=\sum_{n\geq1}B_n(s,t,u,v)z^n=tvz+(st^2v^2+tv+tuv)z^2+\cdots\\
R(s,t,u,v;z)&:=\sum_{n\geq1}R_n(s,t,u,v)z^n=uvz+(stuv^2+tuv+u^2v)z^2+\cdots.
\end{align*}

\begin{lemma}
The following system of algebraic equations hold:
\begin{align}
S&=uvz(1+S|_{v=1})(1+sB),\label{eq:GS1}\\
B&=vz(t+tR|_{v=1}+B|_{v=1})(1+sB)\quad\text{and}\label{eq:GS2}\\
R&=uvz(1+R|_{v=1}+B|_{v=1})(1+sB).\label{eq:GS3}
\end{align}
\end{lemma}
\begin{proof}
We will only show~\eqref{eq:GS2}, the other two equations are similar. For convenience, we use the
convention that $\mathcal{B}_{0}$ and $\mathcal{R}_{0}$ both contain only the empty two-colored Dyck path. 

Each two-colored Dyck path $D=d_1\cdots d_n\in\mathcal{B}_n$ can be decomposed uniquely into a pair
$(D_1,D_2)$ of two-colored Dyck paths, where $D_1=d_2d_3\cdots d_k\in\mathcal{B}_{k-1}\cup\mathcal{R}_{k-1}$
and $D_2=(d_{k+1}-k)(d_{k+2}-k)\cdots (d_n-k)\in\mathcal{B}_{n-k}$ with
$k=\min\{i\in[n]:d_{i+1}=i\text{ or $i=n$}\}$. This decomposition is reversible and satisfies
the following properties: 
\begin{align*}
&\turn(D)=\turn(D_1)+\turn(D_2)+\chi(k\neq n),\\
&\segm(D)=\segm(D_1)+\segm(D_2)+\chi(D_1\in\mathcal{R}_{k-1}),\\
&\Red(D)=\Red(D_1)+\Red(D_2),\\
&\return(D)=1+\return(D_2),
\end{align*}
where $\chi(\mathsf{S})$ equals $1$, if the statement $\mathsf{S}$ is true; and $0$, otherwise. 
Turning this decomposition into generating functions then yields~\eqref{eq:GS2}, as desired. 
\end{proof}

\begin{theorem}[Double Eulerian distribution]\label{dou:eul}Let $S(s,t;z):=S(s,t,1,1;z)$. Then, 
$$
S(s,t;z)=t(z(s-1)+1)S^3(s,t;z)+tz(2s-1)S^2(s,t;z)+z(ts+1)S(s,t;z)+z.
$$
In particular, 
\begin{equation}\label{alg:sepe}
S(1,t;z)=tS^3(1,t;z)+tzS^2(1,t;z)+z(t+1)S(1,t;z)+z
\end{equation}
and
\begin{equation}\label{alg:dist}
S(s,1;z)=(sz-z+1)S^2(s,1;z)+szS(s,1;z)+z,
\end{equation}
where $S(1,t;z):=S(1,t,1,1;z)$ and $S(s,1;z):=S(s,1,1,1;z)$.
\end{theorem}
\begin{remark}
Comparing~\eqref{alg:sepe} with the generating function formula for the descent polynomials of
separable permutations in~\cite[Theorem~5.10]{flz}, we get a new proof of~\eqref{eq:sep}.
Recall that a {\em Schr\"oder $n$-path} is a path in $\N^2$ from $(0,0)$ to $(2n,0)$ using only
the steps $(1,1)$ (up), $(1,-1)$ (down) and $(2,0)$ (flat). An {\em ascent} in a Schr\"oder path
is a maximal string of consecutive up steps. Denote by ${S\!P}_{n}$ the set of all Schr\"oder $n$-paths
and by $\asc(p)$ the number of ascents of $p$. Formula~\eqref{alg:dist} shows (see also the generating
function formula in~\href{https://oeis.org/A090981}{\tt{oeis:A090981}}) that 
$$
\sum_{e\in\I_n(021)}s^{\dist(e)}=\sum_{p\in{S\!P}_{n-1}}s^{\asc(p)},
$$
which confirms a conjecture by Corteel et al.~\cite[Conjecture~2]{cor}.
\end{remark}
\begin{proof}[Proof of Theorem~\ref{dou:eul}] For simplicity, we write $S$, $B$ and $R$ for
$S(s,t,1,1;z)$, $B(s,t,1,1;z)$ and $R(s,t,1,1;z)$, respectively. By~\eqref{eq:GS1} and~\eqref{eq:GS3}, we have 
\begin{align}\label{formula1}
1+sB=\frac{S}{z(1+S)}\quad\text{and}\quad R=\frac{z(1+B)(1+sB)}{1-z(1+sB)}.
\end{align}
Substituting the former into the latter we get $R=(1+B)S$. Plugging this into~\eqref{eq:GS2} yields
\begin{align*}
B&=z(t+tS(1+B)+B)(1+sB)\\
&=\frac{(t+tS+tBS+B)S}{1+S},
\end{align*}
where we apply the first formula in~\eqref{formula1} to get the second equality. Solving the above equality
we obtain $B=\frac{tS+tS^2}{1-tS^2}$. Combining this with the first formula in~\eqref{formula1}, we have 
$$
\frac{tS+tS^2}{1-tS^2}=\frac{S-z(1+S)}{sz(1+S)},
$$
which is simplified to the desired algebraic equation for $S$ after cross-multipying.

The relation \eqref{alg:dist} follows from the factorization of the main equation with $t=1$:
$$
S(s,1;z)=(z(s-1)+1)S^3(s,1;z)+z(2s-1)S^2(s,1;z)+z(s+1)S(s,1;z)+z
$$
factors into
$$
\left(S(s,1;z)+1\right)\left((sz-z+1)S^2(s,1;z)+(sz-1)S(s,1;z)+z\right)=0.
$$
\end{proof}
%%%%%%%%%%%%%%%%%%%%%%%%%%%
\subsection{The $\gamma$-coefficients of ascent polynomial on $\I_n(021)$}
\label{sec:gamma}

In~\cite{cor}, Corteel et al. posed a question for finding a direct combinatorial proof of the palindromicity
of $S_n(t)$ interpreted as $\sum_{e\in\I_n(021)}t^{\asc(e)}$. Although we are unable to provide
such a direct explanation, we could give a combinatorial proof of the $\gamma$-positivity of $S_n(t)$ as
one application of our algorithm $\Psi$ and the classical Foata--Strehl group action~\cite{fsh}. 

An index $i\in[n-2]$ is called a {\em double ascent} of an inversion sequence $e\in\I_n$ if $\{i,i+1\}\subseteq\ASC(e)$. Let 
$$\widetilde{\I_{n,k}}(021):=\{e\in\I_n(021):\asc(e)=k, \text{ $e$ has no double ascents and $e_{n-1}\geq e_n$}\}.$$
\begin{theorem}\label{thm:invseq}
We have
\begin{equation}\label{eq:invseq}
\sum_{e\in\I_n(021)}t^{\asc(e)}=\sum_{k=0}^{\lfloor(n-1)/2\rfloor}|\widetilde{\I_{n,k}}(021)|t^k(1+t)^{n-1-2k}.
\end{equation}
\end{theorem}

The rest of this section is devoted to proving this neat $\gamma$-positivity
expansion. We begin with an action on permutations that
is a modification of the original Foata--Strehl action.

For any $x\in[n]$ and $\pi\in\S_n$,
the {\em$x$-factorization} of $\pi$ is the partition of $\pi$ into the form $\pi=w_1w_2 xw_3w_4$,
where $w_i$'s are intervals of $\pi$ and $w_2$ (resp.~$w_3$) is the maximal contiguous interval
(possibly empty) immediately to the left (resp.~right) of $x$ whose letters are all greater than $x$.
Following Foata and Strehl~\cite{fsh} the action $\varphi_x$ is defined by 
$$
\varphi_x(\pi)=w_1w_3xw_2w_4.
$$
For example, with $\pi=34862571\in\S_8$ and $x=4$, we find that $w_1=3$, $w_2=\emptyset$, $w_3=86$, $w_4=2571$, and we get $\varphi_x(\pi)=38642571$. The map $\varphi_x$ is an involution acting on $\S_n$, and for all $x,y\in[n]$, $\varphi_x$ and $\varphi_y$ commute.

\begin{figure}
\setlength{\unitlength}{1mm}
\begin{picture}(108,38)\setlength{\unitlength}{1mm}
\thinlines
\put(12,15){\dashline{1}(1,0)(40,0)}%3
\put(52,15){\vector(1,0){0.1}}

\put(16,19){\dashline{1}(1,0)(32,0)}%4
\put(48,19){\vector(1,0){0.1}}

\put(40,27){\dashline{1}(-1,0)(-16,0)}%6
\put(24,27){\vector(-1,0){0.1}}

\put(68,23){\dashline{1}(1,0)(16,0)}%5
\put(84,23){\vector(1,0){0.1}}

\put(100,7){\dashline{1}(-1,0)(-96,0)}%1
\put(4,7){\vector(-1,0){0.1}}

\put(0,3){\line(1,1){32}}\put(-2,0){$-\infty$}
\put(12,15){\circle*{1.3}}\put(9,15){$3$}
\put(16,19){\circle*{1.3}}\put(13,19){$4$}
\put(32,35){\circle*{1.3}}\put(29,35){$8$}

\put(32,35){\line(1,-1){24}}
\put(56,11){\circle*{1.3}}\put(57,9){$2$}
\put(40,27){\circle*{1.3}}\put(41,27){$6$}

\put(56,11){\line(1,1){20}}
\put(76,31){\circle*{1.3}}\put(77,31){$7$}
\put(68,23){\circle*{1.3}}\put(65,23){$5$}

\put(76,31){\line(1,-1){28}}\put(103,0){$-\infty$}
\put(100,7){\circle*{1.3}}\put(101,7){$1$}
\end{picture}
\caption{MFS-actions on $34862571$}
\label{valhop}
\end{figure}
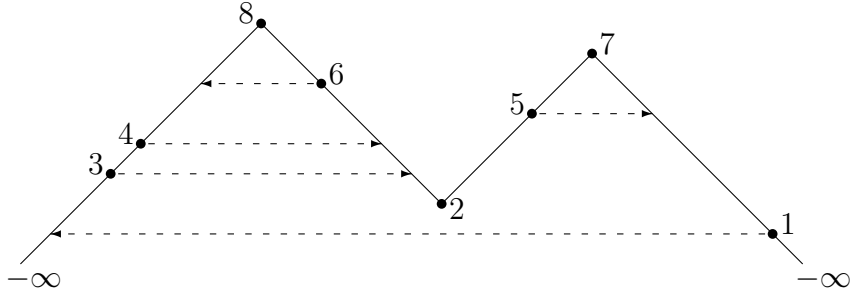

For $\pi\in\S_n$, a value $\pi_i$ ($1\leq i\leq n$) is called a
{\em double ascent} (resp.~{\em double descent, peak, valley}) of $\pi$
if $\pi_{i-1}<\pi_i<\pi_{i+1}$ (resp.~$\pi_{i-1}>\pi_i>\pi_{i+1}$,
$\pi_{i-1}<\pi_i>\pi_{i+1}$, $\pi_{i-1}>\pi_i<\pi_{i+1}$),
where we use the convention $\pi_0=\pi_{n+1}=-\infty$.
For our purpose, we slightly modify the action $\varphi_x$ to be 
$$
\varphi'_x(\pi):=
\begin{cases}
\varphi_x(\pi),&\text{if $x$ is a double ascent or a double descent of $\pi$};\\
\pi,& \text{if $x$ is a peak or a valley of $\pi$.}
\end{cases}
$$
Again all $\varphi'_x$'s are involutions and commute. Therefore, for any $S\subseteq[n]$ we can define the function $\varphi'_S:\S_n\rightarrow\S_n$ by $\varphi'_S=\prod_{x\in S}\varphi'_x$, where multiplication is the composition of functions. Hence the group $\Z_2^n$ acts on $\S_n$ via the function $\varphi'_S$. This action is called the {\em Modified Foata--Strehl action} (MFS-action for short) and has a nice visualization as depicted in Figure~\ref{valhop}. It should be noted that this action is different  with the modification of Foata-Strehl action defined in~\cite{lz,pe}, where they take the convention $\pi_0=\pi_{n+1}=+\infty$.

For any $\mathcal{S}\subseteq\S_n$, let 
$$
\widetilde{\mathcal{S}_{n,k}}:=\{\pi\in \mathcal{S}:\des(\pi)=k,\text{ $\pi$ has no double descents}\}.
$$
\begin{lemma}\label{act:inv}
 Suppose $\mathcal{S}\subseteq\S_n$ is invariant under the MFS-action. Then,
\begin{equation}\label{gam:int}
\sum_{\pi\in\mathcal{S}}t^{\des(\pi)}=\sum_{k=0}^{\lfloor(n-1)/2\rfloor}|\widetilde{\mathcal{S}_{n,k}}|t^k(1+t)^{n-1-2k}.
\end{equation}
\end{lemma}
\begin{proof}
For each $\pi\in\mathcal{S}$, let $\Orb(\pi)=\{g(\pi): g\in\Z^n_2\}$ be the orbit of $\pi$ under the MFS-action. Note that if $x$ is a double descent of $\pi$, then $x$ is a double ascent of $\varphi'(\pi)$. Thus, there exists a unique permutation in $\Orb(\pi)$ which has no double descent. Moreover, if we let $\bar{\pi}$ be such a unique element in $\Orb(\pi)$, then we have 
$$
\sum_{\sigma\in\Orb(\pi)}t^{\des(\sigma)}=t^{\des(\bar{\pi})}(1+t)^{\da(\bar{\pi})}=t^{\des(\bar{\pi})}(1+t)^{n-1-2\des(\bar{\pi})},
$$
where $\da(\bar{\pi})$ denotes the number of {\bf d}ouble {\bf a}scent of $\bar{\pi}$ and the second equality follows from the simple fact that $\da(\bar{\pi})=n-1-2\,\des(\bar{\pi})$. Summing over all orbits of $\mathcal{S}$ under the MFS-action gives~\eqref{gam:int}. Consider for instance $\pi=34862571$ in Figure~\ref{valhop},  then
$\bar{\pi}=13468257$, $\da(\bar{\pi})=5$, $n=8$ and $\des(\bar{\pi})=1$.
\end{proof}

\begin{lemma}\label{patt:inv}
The set $\S_n(2413,4213)$ is invariant under the MFS-action.
\end{lemma}
\begin{proof}
First we show that if $\pi\in\S_n$ contains the pattern $2413$, then $\varphi'_x(\pi)$ 
must contain the pattern $2413$ or $4213$ for any $x\in[n]$. Since $\pi$ contains the 
pattern $2413$, there exist $i<j<k<l$ such that $\pi_i\pi_j\pi_k\pi_l$ is order 
isomorphic to $2413$. Let us consider the largest peak $\pi_{j'}$ with $i<j'<k$. Note 
that such a peak $\pi_{j'}$ always exists because $\pi_i<\pi_j>\pi_l$ and obviously
$\pi_{j'}\geq\pi_j$. Then we consider the smallest valley $\pi_{k'}$ with $j'<k'<l$, 
which also exists (as $\pi_{j'}>\pi_k<\pi_l$) and is less than or equal to
$\pi_k$. Now the (new) subsequence $\pi_i\pi_{j'}\pi_{k'}\pi_l$ is again a $2413$
pattern in $\pi$, 
but $\pi_{j'}$ and $\pi_{k'}$ are respectively a peak and a valley of $\pi$. It is clear 
from the definition of $\varphi'$ that for any $x\in[n]$, the subsequence with letters
$\{\pi_i,\pi_{j'},\pi_{k'},\pi_l\}$ is either a pattern $2413$ or $4213$ in
$\varphi'_x(\pi)$.

Similarly, we can show that if $\pi\in\S_n$ contains the pattern $4213$, then
$\varphi'_S(\pi)$ must contain the pattern $2413$ or $4213$ for any $S\subseteq[n]$.
This completes the proof of the lemma.
\end{proof}

\begin{proof}[Proof of Theorem~\ref{thm:invseq}] It follows from
Lemmas~\ref{patt:inv} and~\ref{act:inv} that 
\begin{equation}\label{expan:per}
\sum_{\pi\in\S_n(2413,4213)}t^{\des(\pi)}=\sum_{k=0}^{\lfloor(n-1)/2\rfloor}|\widetilde{\S_{n,k}}(2413,4213)|t^k(1+t)^{n-1-2k},
\end{equation}
where
$$
\widetilde{\S_{n,k}}(2413,4213)=\{\pi\in \S_n(2413,4213):\des(\pi)=k,\text{ $\pi$ has no double descents}\}.
$$
Applying $\Psi^{-1}$ to both sides of expansion~\eqref{expan:per}
yields~\eqref{eq:invseq} because of (1) in Lemma~\ref{lem:lrm}.
\end{proof}

%%%%%%%%%%%%%%%%%%%%%%%%%%%
\section{On the $\des$-Wilf equivalences for Schr\"oder classes}
\label{sec:wilf}

For two sets of patterns $\Pi$ and $\Pi'$, we write $\Pi\sim_{\des}\Pi'$ if 
$$
\sum_{\sigma\in\S_n(\Pi)}t^{\des(\sigma)}=\sum_{\sigma\in\S_n(\Pi')}t^{\des(\sigma)}.
$$
Similarly, we define $\Pi\sim_{\DES}\Pi'$. We say that $\Pi$ is
{\em$\des$-Wilf (resp.~$\DES$-Wilf) equivalent} to $\Pi'$ if $\Pi\sim_{\des}\Pi'$
(resp.~$\Pi\sim_{\DES}\Pi'$). Letting $s=1$ in~\eqref{restrict}, we obtain
the restricted version of~\eqref{des:asc}: 
\begin{equation}\label{restr:des}
\sum_{\pi\in\S_n(2413,4213)}t^{\DES(\pi)}=\sum_{e\in\I_n(021)}t^{\ASC(e)}.
\end{equation} 
Motivated by the above equidistribution and~\eqref{eq:sep}, we prove in this section
the following refined Wilf-equivalences.

\begin{theorem}\label{wilf4}
We have the following refined Wilf-equivalences: 
$$
(2413,3142)\sim_{\des}(2413,4213)\sim_{\Des}(2314,3214)\sim_{\Des}(3412,4312).
$$
\end{theorem}
\begin{remark}
All the above Wilf-equivalences are known in the works of West~\cite{we} and
Kremer~\cite{kre}. A simple Sage program shows that, among all permutation classes
avoiding two patterns of length $4$, these are all the nontrivial classes which are
$\des$-Wilf equivalent to the classes of separable permutations. Some nice refined
Wilf-equivalences can be found in Elizald and Pak~\cite{eli}, Dokos et al.~\cite{sag}
and Bloom~\cite{blo}.
\end{remark}

\begin{proof}[Proof of Theorem~\ref{wilf4}] 
We first construct a bijection
$\Phi:\S_n(2413,4213)\rightarrow\I_n(021)$,
which transforms the statistic ``$\DES$'' to ``$\ASC$'' and therefore provides another
simple proof of equidistribution~\eqref{restr:des}.
%\blue{Note that bijection $\Theta$ proving equidistribution~\eqref{des:asc}
%does not work for equidistribution~\eqref{restr:des}.}
Our bijection $\Phi$ will be defined recursively, which can be considered as a restricted version of $\Theta$
for equidistribution~\eqref{des:asc}.

For any $k\in[n]$, let us define the inserting operator $T_k:\S_{n-1}\rightarrow\S_n$ by
$$
T_k(\pi)=\pi'_1\pi'_2\ldots\pi'_{n-1}k\quad\text{for each $\pi\in\S_{n-1}$},
$$
where $\pi'_i=\pi_i+\chi(\pi_i\geq k)$ for $1\leq i\leq n-1$. If $\pi\in\S_{n-1}(2413,4213)$, then introduce the set of {\em{\bf\em ava}ilable inserting values} of $\pi$ as
$$
\AVA(\pi):=\{k\in[n]: T_k(\pi)\in\S_{n}(2413,4213)\}=\{k_1>k_2>\cdots\},
$$
where $k_j$ is called the {\em$j$-th available inserting value} of $\pi$. Clearly,
$n$ and $1$ are respectively the first and the last available inserting values of $\pi$.
For example, we have $\AVA(3142)=\{5>3>2>1\}$.

Before we complete the proof, we need the following lemma:
\begin{lemma}\label{insert}Suppose $\pi\in\S_{n-1}(2413,4213)$ with 
$$
\AVA(\pi)=\{n=k_1>k_2>\cdots>k_m=1\}.
$$
Then, for  $1\leq j\leq m$,
$$
\AVA(T_{k_j}(\pi))=\{n+1\geq k_j+1>k_j>k_{j+1}>\cdots>k_m=1\}.
$$
\end{lemma}
\begin{proof}
If $k$ is less than or equal to $k_j$, then the subsequence of $T_k(T_{k_j}(\pi))$ composed of letters
$[n+1]\setminus\{k_j+1\}$ is order isomorphic to $T_k(\pi)$. Therefore, for all $k\leq k_j$, we have
$k\in\AVA(T_{k_j}(\pi))$ if and only if $k\in\AVA(\pi)$. On the other hand,
if $k$ is greater than $k_j$, then the penultimate and the ultimate letters ``$k_jk$'' can play
the role of ``$13$'' in a length $4$ subsequence of $T_k(T_{k_j}(\pi))$ if and only if $k_j+1<k\leq n$. 
This completes the proof of the lemma, since $\{k_j+1,n+1\}\subseteq\AVA(T_{k_j}(\pi))$.
\end{proof}

We now continue to prove Theorem~\ref{wilf4}.
Using the above lemma, we define $\Phi$ recursively as follows. Suppose
$\Phi$ is defined for $\S_{n-1}(2413,4213)$. Any permutation $\sigma\in\S_n(2413,4213)$
can be expressed uniquely as $\sigma=T_k(\pi)$ for some $\pi\in\S_{n-1}(2413,4213)$
and $k\in\AVA(\pi)$. If $k$ is the $j$-th available inserting value of $\pi$ and
$(e_1,\ldots,e_{n-1})=\Phi(\pi)\in\I_{n-1}(021)$, then define
$\Phi(\sigma)=(e_1,\ldots,e_{n-1},m_j)$, where $m_j$ is the $j$-th smallest integer
in the set
$$
\{0,m,m+1,\ldots,n-1\}\text{ with $m=\max\{e_1,\ldots,e_{n-1}\}$}.
$$
For example, we have $\Phi(5164372)=(0,1,0,2,3,0,5)$. 
Using Lemma~\ref{insert}, it is routine to check by induction on $n$ that
$\Phi:\S_n(2413,4213)\rightarrow\I_n(021)$ as defined above is a bijection, which
transforms the statistic ``$\DES$'' to ``$\ASC$''.

It turns out that a similar construction can be applied to show
$$ 
\sum_{\pi\in\S_n(2314,3214)}t^{\DES(\pi)}=\sum_{e\in\I_n(021)}t^{\ASC(e)}\quad\text{and}\quad\sum_{\pi\in\S_n(3412,4312)}t^{\DES(\pi)}=\sum_{e\in\I_n(021)}t^{\ASC(e)},
$$
which completes the proof of the theorem. The details are left to interested readers. 
\end{proof}

\begin{remark}
Actually, bijection $\Phi:\S_n(2413,4213)\rightarrow\I_n(021)$
transforms the quadruple $(\DES,\LMA,\LMI,\RMA)$ to $(\ASC,\ZERO,\EMA,\RMI)$
as the natural bijection $\Theta:\S_n\rightarrow\I_n$ does. 
\end{remark}

%%%%%%%%%%%%%%%%%
%%%%%%%%%%%%%%%%%%%%%%%%%%%%%%%%%
\section{Final remarks}
\label{sec:fin}

In view of Theorem~\ref{thm:sex}, a natural question is to see if there exists an analog extension of Foata's equidistribution~\eqref{dist:asc} regarding our set-valued distributions defined for all elements of $\I_n$ or $\S_n$. The following result answers this question. 

\begin{theorem}\label{thm:gen}
For $n\geq1$, the quadruple distribution $(\DIST,\ASC,\ZERO,\EMA)$ on $\I_n$ is equidistributed with $(\VID,\DES,\LMA,\LMI)$ on $\S_n$. In particular, 
\begin{equation}
\sum_{e\in\I_n}s^{\DIST(e)}t^{\ASC(e)}=\sum_{\pi\in\S_n}s^{\VID(\pi)}t^{\DES(\pi)}.
\end{equation}
\end{theorem}

Theorem~\ref{thm:gen} can be proved via the principle of {\em Inclusion-Exclusion} (essentially Aas' approach~\cite{aas}) and a simple (but crucial) bijection, which will appear in~\cite{kl}.
Very recently, during our preparation of this paper, Baril and Vajnovszki~\cite{bv}  constructed a new coding $b:\S_n\rightarrow\I_n$ such that for each $\pi\in\S_n$:
$$
(\VID,\DES,\LMA,\LMI,\RMA)\pi=(\DIST,\ASC,\ZERO,\EMA,\RMI)b(\pi),
$$
which provides a direct bijective proof of  Theorem~\ref{thm:gen}.
Interestingly, we have been able to show that their coding $b$ restricts to a bijection between $\S_n(2413,4213)$ and $\I_n(021)$. This restricted $b$ does not transform ``$\RMI$'' to ``$\EXPO$'', while  our bijection $\Psi^{-1}$ in Theorem~\ref{thm:sex} does.

\subsection*{Acknowledgement} The first author's research was supported by the National Science Foundation of China grant 11501244.

\end{document}